\documentclass{article}

\usepackage[utf8]{inputenc}
\usepackage{amssymb, amsmath}
\usepackage{amsthm}
\usepackage{amsfonts} 
\usepackage{booktabs}
\newtheorem{thm}{Theorem}[section]
\newtheorem{lem}{Lemma}[section]

\newtheorem{dfn}{Definition}[section]

\usepackage[parfill]{parskip}

\title{Canonical forms for a class of pairs of commuting nilpotent matrices under simultaneous similarity}
\author{Jiuzhao Hua} 
\begin{document}
\date{\vspace{-0.5cm}}\date{} 
% Toggle commenting to test
\maketitle
\begin{abstract}
We present canonical forms for all indecomposable pairs $(A,B)$ of commuting nilpotent matrices over an arbitrary field under simultaneous similarity, 
where $A$ is the direct sum of two Jordan blocks with distinct sizes. We also provide the transformation matrix $X$ such that $(A, X^{-1}BX)$ is in its canonical form. 
\end{abstract}

\section{Introduction}

Let $\mathbb{F}$ be a field and let $n$ be a non-negative integer. Two pairs of $n\times n$ matrices $(A,B)$ and $(C,D)$ over $\mathbb{F}$ are said to be \textit{similiar} if there exists a non-singular 
matrix $X$ over $\mathbb{F}$ such that $X^{-1}AX = C$ and $X^{-1}BX = D$. A pair $(A,B)$ is said to be a \textit{commuting nilpotent} pair if $AB=BA$, $A^n=0$ and $B^n=0$. 
The direct sum of two matrices $M$ and $N$, denoted by $M\oplus N$, is the following block matrix:
\[
\left[
\begin{array}{cc}
M & 0 \\ 
0 & N \\
\end{array}
\right].
\]

We say that a pair $(A,B)$ of $n\times n$ matrices over a field $\mathbb{F}$ is \textit{decomposable} if it is similar to a pair of the form $(A_1\oplus A_2, B_1 \oplus B_2)$, where $A_1$ and $B_1$ are square matrices of the same size.

Let $\textrm{Mat}(n,\mathbb{F})$ be the matrix algebra of order $n$ over $\mathbb{F}$, consisting of all $n\times n$ matrices over $\mathbb{F}$, and let $\mathrm{GL}(n,\mathbb{F})$ be the General Linear Group of order $n$ over $\mathbb{F}$, consisting of all non-singular $n\times n$ matrices over $\mathbb{F}$. Given a matrix $A\in\textrm{Mat}(n,\mathbb{F})$, the \textit{stabilizer group} of $A$ in $\mathrm{GL}(n,\mathbb{F})$, denoted by $\mathrm{Stab}(A)$, is defined as follows:
\[
\mathrm{Stab}(A) = \big\{X\in\mathrm{GL}(n,\mathbb{F}) : X^{-1} A X = A \big\},
\]
and the \textit{nilpotent commutator} of $A$, denoted by $\mathrm{NilC}(A)$, is defined as follows:
\[
\mathrm{NilC}(A) = \big\{B\in\textrm{Mat}(n,\mathbb{F}) : B^n=0, AB=BA \big\}.
\]
Thus, we have the following group action:
\begin{align*}\label{group action}
\mathrm{Stab}(A) \times \mathrm{NilC}(A) &\to \mathrm{NilC}(A) \\
(X, B)\,\,\,\,\,\,\,\,\,\,\,\,\,\,\,\, &\mapsto  X^{-1}BX.
\end{align*}
It follows that two pairs $(A,B)$ and $(A,B')$ are similar if and only if $B$ and $B'$ are in the same orbit under the above action.

It is well-known that classifying all commuting nilpotent pairs up to simultaneous similarity is considered a 'wild' problem, as it contains the problem of classifying pairs of matrices up to simultaneous similarity. 
Extensive research has been conducted in this field over the past four decades, and interested readers can refer to \cite{VB-2001}, \cite{BI-2008}, and \cite{BFPS-2021}.

The matrix that is commonly referred to as the Jordan matrix or Jordan block of order $n$ with eigenvalue 0 is given by the following expression:
\[
\newcommand*{\temp}{\multicolumn{1}{|}{}}
J_n = 
\left[
\begin{array}{ccccc}
0 & 1 & 0 & \dots & 0 \\ 
0 & 0 & 1 & \dots & 0 \\ 
\vdots & \vdots & \vdots & \ddots & \vdots \\
0 & 0 & 0 & \dots & 1 \\
0 & 0 & 0 & \dots & 0 \\
\end{array}
\right]_{n\times n}.
\]
According to the Jordan normal form theorem, any nilpotent matrix can be expressed as a direct sum of Jordan matrices.

This paper aims to classify a particular class of commuting nilpotent pairs $(A,B)$ under simultaneous similarity, where $A$ is similar to a direct sum of two Jordan blocks with distinct sizes. 
To achieve this, we use a matrix reduction process similar to the Belitskii algorithm described in \cite{GB-2000}.

This paper is structured as follows. In Section 2, we introduce the concept of canonical rank for nilpotent matrices that commute with $J_m\oplus J_n$, where $m>n$. 
We demonstrate that a matrix with a particular canonical rank must have a certain structure and that the canonical rank is invariant under the group action mentioned above.
In Section 3, we present the canonical forms for all indecomposable nilpotent commuting pairs $(J_m\oplus J_n, B)$ and provide the transformation 
matrix $X$ such that $X^{-1}BX$ is in its canonical form.
In the Appendix, we list all possible canonical forms for the special case where $m=6$ and $n=4$.

\section{The canonical ranks}

\begin{dfn} 
Let $M=[a_{ij}]$ be an $m\times n$ matrix. For $1\leq i\leq m$ and $1\leq j\leq n$, the arm length of the element $a_{ij}$ is defined as follows: 
\begin{itemize}
\item If $m\geq n$, then the arm length of $a_{ij}$ is $j-i$,
\item If $m<n$, then the arm length of $a_{ij}$ is $j-i-(n-m)$.
\end{itemize}
\end{dfn}

\begin{dfn}
A matrix $M = [a_{ij}]$ of order $m\times n$ is called a TA-matrix if it satisfies the following conditions:
\begin{itemize}
\item If $a_{ij}$ and $a_{kl}$ have the same arm length, then $a_{ij} = a_{kl}$,
\item If the arm length of $a_{ij}$ is less than 0, then $a_{ij} = 0$.
\end{itemize}
\end{dfn}

By the above definitions, a TA-matrix of order $m\times n$ has the following form when $m\ge n$:
\[
\newcommand*{\temp}{\multicolumn{1}{|}{}}
\left[
\begin{array}{ccccc}
a_0 & a_1 & \dots & a_{n-2} & a_{n-1} \\ 
0 & a_0 & \dots & a_{n-3} & a_{n-2} \\ 
\vdots & \vdots & \ddots & \vdots & \vdots \\
0 & 0 & \dots & a_0 & a_1 \\
0 & 0 & \dots & 0 & a_0 \\
\cline{1-5}
0 & 0 & \dots & 0 & 0 \\
\vdots & \vdots & \ddots & \vdots & \vdots \\
0 & 0 & \dots & 0 & 0 \\
\end{array}
\right]_{m \times n},
\]
or the following form when $m< n$:
\[
\newcommand*{\temp}{\multicolumn{1}{|}{}}
\left[
\begin{array}{ccccccccc}
0 & \dots & 0 & \temp & a_0 & a_1 & \dots & a_{m-2} & a_{m-1} \\ 
0 & \dots & 0 & \temp & 0 & a_0 & \dots & a_{m-3} & a_{m-2} \\ 
\vdots & \ddots & \vdots & \temp & \vdots & \vdots & \ddots & \vdots & \vdots \\
0 & \dots & 0 & \temp & 0 & 0 & \dots & a_0 & a_1 \\
0 & \dots & 0 & \temp & 0 & 0 & \dots & 0 & a_0 \\
\end{array}
\right]_{m \times n}. 
\]

\begin{thm}[Turnbull and Aitken \cite{TA-1948}]\label{T-A Thm}
Let $\lambda=[\lambda_1, \lambda_2, \dots, \lambda_s]$ be a partition with $\lambda_1\ge\lambda_2\ge\cdots\ge\lambda_s\ge 1$ and $A=J_{\lambda_1} \oplus \cdots \oplus J_{\lambda_s}$. 
Then any matrix $B$ that commutes with $A$ can be written as an $s\times s$ block matrix in the following form:
\begin{equation}\label{TA block matrix}
\newcommand*{\temp}{\multicolumn{1}{|}{}}
\left[
\begin{array}{cccc}
B_{11} & B_{12} & \dots & B_{1s} \\ 
B_{21} & B_{22} & \dots & B_{2s} \\ 
\vdots & \vdots & \ddots & \vdots \\
B_{s1} & B_{s2} & \dots & B_{ss} \\
\end{array}
\right],
\end{equation}
where each submatrix $B_{ij}$ is a TA-matrix of order $\lambda_i\times \lambda_j$ for $1\le i,j \le s$.
\end{thm}

For any block matrix of the form (\ref{TA block matrix}), the arm length of an element is defined as the arm length of the element relative to the block it resides in. 
If $A$ is the Jordan block matrix $J_m \oplus J_n$ with $s = m - n \geq 1$ and $B$ is a square matrix that commutes with $A$, then by Theorem \ref{T-A Thm}, $B$ can be written as a block matrix of the following form:
\begin{equation}\label{B full form}
\arraycolsep=2.5pt\def\arraystretch{1.2}
\newcommand*{\temp}{\multicolumn{1}{|}{}}
\left[
\begin{array}{ccccccccccccccc}
a_0 & a_1 & \cdots & a_{s-1} & a_s & a_{s+1} & \cdots & a_{m-2} & a_{m-1} & \temp & b_0 & b_1 & \cdots & b_{n-2} & b_{n-1} \\ 
0 & a_0 & \cdots & a_{s-2} & a_{s-1} & a_s & \cdots & a_{m-3} & a_{m-2} & \temp &  0 & b_0 & \cdots & b_{n-3} & b_{n-2} \\ 
\vdots & \vdots & \ddots & \vdots  & \vdots & \vdots & \ddots & \vdots & \vdots & \temp & \vdots & \vdots & \ddots & \vdots & \vdots \\
0 & 0 & \cdots & a_0 & a_1 & a_2 & \cdots & a_{s} & a_{s+1} & \temp  &  0 & 0 & \cdots & b_0 & b_{1} \\
0 & 0 & \cdots & 0  & a_0 & a_1 & \cdots & a_{s-1} & a_s & \temp  &  0 & 0 & \cdots & 0 & b_{0} \\
0 & 0 & \cdots & 0 & 0 & a_0 & \cdots & a_{s-2} & a_{s-1} & \temp  &  0 & 0 & \cdots & 0 & 0 \\
\vdots & \vdots & \ddots & \vdots  & \vdots & \vdots & \ddots & \vdots & \vdots & \temp & \vdots & \vdots & \ddots & \vdots & \vdots \\
0 & 0 & \cdots & 0 & 0 & 0 & \cdots & a_0 & a_1 & \temp  &  0 & 0 & \cdots & 0 & 0 \\
0 & 0 & \cdots & 0 & 0 & 0 & \cdots & 0 & a_0 & \temp & 0 & 0 & \cdots & 0 & 0 \\
\cline{1-15}
0 & 0 & \cdots & 0 & c_0 & c_1 & \cdots & c_{n-2} & c_{n-1} & \temp & d_0 & d_1 & \cdots & d_{n-2} & d_{n-1} \\ 
0 & 0 & \cdots & 0 & 0 & c_0 & \cdots &  c_{n-3} & c_{n-2} & \temp &  0 & d_0 & \dots & d_{n-3} & d_{n-2} \\ 
\vdots & \vdots & \ddots & \vdots & \vdots & \vdots & \ddots & \vdots & \vdots & \temp & \vdots & \vdots & \ddots & \vdots & \vdots \\
0 & 0 & \cdots & 0 & 0 & 0 & \cdots & c_0 & c_1 & \temp  &  0 & 0 & \dots & d_{0} & d_{1} \\
0 & 0 & \cdots & 0 & 0 & 0 & \cdots & 0 &  c_0 & \temp & 0 & 0 & \cdots & 0 & d_{0} \\
\end{array}
\right].
\end{equation}

If $B$ is nilpotent, then we have $a_0=0$ and $d_0=0$. Therefore, every matrix of the form (\ref{B full form}) is uniquely determined 
by its values in row 1 and row $m+1$, and can be represented by the following $2 \times 2$ block matrix:
\begin{equation}\label{B short form}
\arraycolsep=2.5pt\def\arraystretch{1.2}
\newcommand*{\temp}{\multicolumn{1}{|}{}}
\left[
\begin{array}{ccccccccccccccc}
a_0 & a_1 & \cdots & a_{s-1} & a_s & a_{s+1} & \cdots & a_{m-2} & a_{m-1} & \temp & b_0 & b_1 & \cdots & b_{n-2} & b_{n-1} \\ 
\cline{1-15}
0 & 0 & \cdots & 0 & c_0 & c_1 & \cdots & c_{n-2} & c_{n-1} & \temp & d_0 & d_1 & \cdots & d_{n-2} & d_{n-1} \\ 
\end{array}
\right].
\end{equation}
Conversely, every matrix of the form (\ref{B short form}) can be expanded to a unique matrix of the form (\ref{B full form}).
Hence, we may use the forms (\ref{B full form}) and (\ref{B short form}) interchangeably throughout the rest of this paper.

\begin{lem}\label{B indecom form}
Consider the direct sum $A=J_{m}\oplus J_n$, where $s=m - n$ with $s\geq 1$. Let $B$ be an element of the nilpotent commutator
$\mathrm{NilC}(A)$ that can be expressed in the following form:
\[
\arraycolsep=4pt\def\arraystretch{1.2}
\newcommand*{\temp}{\multicolumn{1}{|}{}}
\left[
\begin{array}{cccccccccccc}
0 & a_1 & \cdots  & a_{s-1}  & a_{s} & \cdots & a_{m-1} & \temp & b_0 & b_1 & \cdots  & b_{n-1} \\ 
\cline{1-12}
0 & 0 & \cdots & 0 & c_0 & \cdots  & c_{n-1} & \temp & 0 & d_1 & \cdots & d_{n-1} \\ 
\end{array}
\right].
\]
Suppose that $(A,B)$ is indecomposable. Then, there exists an integer $r$ such that $1 \le r \le n$ and the following conditions hold: 
$(b_{n-r}, c_{n-r}) \ne (0, 0)$, $(b_i, c_i) = (0, 0)$ for $0\le i < n-r$, and $d_i=a_i$ for $1\le i \le n-r$. 
\end{lem}
\begin{proof}
Since $(A,B)$ is indecomposable, at least one of the following pairs is not equal to $(0, 0)$:
\[
(b_0,c_0), (b_1,c_1),\cdots,(b_{n-1},c_{n-1}).
\]
Suppose that $k$ is the smallest integer such that $(b_k, c_k) \ne (0,0)$. Let $r=n-k$, then we have 
\[
\arraycolsep=2.5pt\def\arraystretch{1.2}
\newcommand*{\temp}{\multicolumn{1}{|}{}}
B =
\left[
\begin{array}{cccccccccccccccc}
0 & a_1 & \cdots & \cdot  & a_{m-r} &  \cdots & a_{m-1} & \temp & 0 & 0  & \cdots & 0 & b_{n-r}  & b_{n-r+1} & \cdots & b_{n-1} \\ 
\cline{1-16}
0 & 0 & \cdots & 0 & c_{n-r} &  \cdots & c_{n-1} & \temp & 0 & d_1 & \cdots & \cdot & d_{n-r} & d_{n-r+1} & \cdots & d_{n-1} \\ 
\end{array}
\right].
\]

For any $X\in\mathrm{Stab}(A)$, $(A, X^{-1}BX)$ is similar to $(A,B)$.
If $d_1\ne a_1$, then the following pairs can be progressively reduced to $(0,0)$ by applying appropriate conjuations on $B$:
\[
(b_0,c_0), (b_1,c_1),\cdots,(b_{n-1},c_{n-1}),
\]
which contradicts the assumption that $(A, B)$ is indecomposable. Thus, we have $d_1= a_1$. Similarly, by applying the same arguments to the pairs 
$(a_i,d_i)$ for $2 \le i \le n-r$, we obtain the conclusion that $a_i=d_i$ for $2 \le i \le n-r$. Thus, $B$ has the following form:
\begin{equation}
\arraycolsep=2.5pt\def\arraystretch{1.2}
\newcommand*{\temp}{\multicolumn{1}{|}{}}
\left[
\begin{array}{cccccccccccccccc}
0 & a_1 & \cdots & \cdot  & a_{m-r} &  \cdots & a_{m-1} & \temp & 0 & 0  & \cdots & 0 & b_{n-r}  & b_{n-r+1} & \cdots & b_{n-1} \\ 
\cline{1-16}
0 & 0 & \cdots & 0 & c_{n-r} &  \cdots & c_{n-1} & \temp & 0 & a_1 & \cdots & \cdot & a_{n-r} & d_{n-r+1} & \cdots & d_{n-1} \\ 
\end{array}
\right].
\end{equation}
\end{proof}

The integer $r$ from Lemma \ref{B indecom form} is called the \textit{canonical rank} of $B$. If $(A,B)$ is decomposable, then the canonical rank of $B$ is defined to be 0. 
The next lemma shows that the canonical rank is invariant under the group action of $\mathrm{NilC}(A)/\mathrm{Stab}(A)$.

\begin{lem}\label{rank inv}
Consider the direct sum $A=J_{m}\oplus J_n$, where $s=m - n$ with $s\geq 1$. Let $B \in \mathrm{NilC}(A)$ and $X\in\mathrm{Stab}(A)$. 
Then $X^{-1}BX$ and $B$ have the same canonical rank.
\end{lem}
\begin{proof} 
Suppose that $B$ has canonical rank $r$ with $r\ge 1$. According to Lemma \ref{B indecom form}, $B$ has the following for:
\[
\arraycolsep=2.5pt\def\arraystretch{1.2}
\newcommand*{\temp}{\multicolumn{1}{|}{}}
\left[
\begin{array}{cccccccccccccccc}
0 & a_1 & \cdots & \cdot  & a_{m-r} &  \cdots & a_{m-1} & \temp & 0 & 0  & \cdots & 0 & b_{n-r}  & b_{n-r+1} & \cdots & b_{n-1} \\ 
\cline{1-16}
0 & 0 & \cdots & 0 & c_{n-r} &  \cdots & c_{n-1} & \temp & 0 & a_1 & \cdots & \cdot & a_{n-r} & d_{n-r+1} & \cdots & d_{n-1} \\ 
\end{array}
\right]
\]
with $(b_{n-r}, c_{n-r}) \ne (0,0)$. Theorem \ref{T-A Thm} implies that 
\[
\arraycolsep=3pt\def\arraystretch{1.2}
\newcommand*{\temp}{\multicolumn{1}{|}{}}
X = 
\left[
\begin{array}{ccccccccccccccc}
x_0 & x_1 & \cdots & x_{s-1} & x_s & \cdots  & x_{m-1} & \temp & y_0 & y_1 & \cdots & y_{n-1} \\ 
\cline{1-12}
0 & 0 & \cdots & 0 & z_0 & \cdots & z_{n-1}  & \temp & w_0 & w_1 & \cdots & w_{n-1} \\ 
\end{array}
\right]
\]
with $x_0\ne 0$ and $w_0\ne 0$. It can be shown that $X^{-1}BX$ has the following form:
\[
\arraycolsep=2.5pt\def\arraystretch{1.2}
\newcommand*{\temp}{\multicolumn{1}{|}{}}
\left[
\begin{array}{cccccccccccccccc}
0 & a_1 & \cdots & a_{m-r-1}  & a'_{m-r} & \cdots & a'_{m-1} & \temp & 0 & 0  & \cdots & 0 & b'_{n-r}  & b'_{n-r+1} & \cdots & b'_{n-1} \\ 
\cline{1-16}
0 & 0 & \cdots & 0 & c'_{n-r} &  \cdots & c'_{n-1} & \temp & 0 & a_1 & \cdots & \cdot & a_{n-r} & d'_{n-r+1} & \cdots & d'_{n-1} \\ 
\end{array}
\right],
\]
where $(b'_{n-r}, c'_{n-r}) = (b_{n-r}w_0,  c_{n-r} w_0^{-1}) \ne (0, 0)$. Thus $X^{-1}BX$ has canonical rank $r$.
\end{proof}

\section{The canonical forms}

The main results of this paper are presented in the following two theorems.

\begin{thm} \label{thm can}
Consider positive integers $m$ and $n$ with $m>n$, and let $(A,B)$ be an indecomposable pair of commuting nilpotent matrices 
of order $(m+n)\times (m+n)$ over $\mathbb{F}$. Assume that $A$ is similar to the Jordan block matrix $J_{m}\oplus J_n$. 
Then, $(A,B)$ is similar to a pair that is in one of the following forms:
\begin{equation}\label{canonical form}
(J_m\oplus J_n, B_{m,n,r}) \text{ or } (J_m\oplus J_n, B'_{m,n,r}),
\end{equation}
where $B_{m,n,r}$ has the following form:
\[
\arraycolsep=3.8pt\def\arraystretch{1.2}
\newcommand*{\temp}{\multicolumn{1}{|}{}}
\left[
\begin{array}{ccccccccccccccccc}
0 & a_1 & \cdots & a_{m-r-1} & 0 & 0 & \cdots  & 0 & \temp & 0 & 0 & \cdots & 0 & b_{n-r} & b_{n-r+1} & \cdots & b_{n-1}  \\ 
\cline{1-17}
0 & 0 & \cdots & 0 & 1 & 0 & \cdots & 0 & \temp & 0 & a_1 & \cdots & \cdot & a_{n-r} & d_{n-r+1} & \cdots & d_{n-1} \\ 
\end{array}
\right],
\]
and $B'_{m,n,r}$ has the following form:
\[
\arraycolsep=2.5pt\def\arraystretch{1.2}
\newcommand*{\temp}{\multicolumn{1}{|}{}}
\left[
\begin{array}{ccccccccccccccccc}
0 & a_1 & \cdots & a_{m-r-1} & 0 & 0 & \cdots  & 0 & \temp & 0 & 0 & \cdots & 0 & 1 & 0 & \cdots & 0  \\ 
\cline{1-17}
0 & 0 & \cdots & 0 & 0 & c_{n-r+1} & \cdots & c_{n-1} & \temp & 0 & a_1 & \cdots & \cdot & a_{n-r} & d_{n-r+1} & \cdots & d_{n-1} \\ 
\end{array}
\right].
\]
where $1\le r \le n$ and $a_i$ $(1\le i \le m-r-1)$, $b_i$ $(n-r\le i \le n-1)$, $c_i$ $(n-r+1\le i \le n-1)$ and $d_i$ $(n-r+1\le i \le n-1)$ are independent elements in $\mathbb{F}$. 
Furthermore, two pairs in the canonical forms as shown in (\ref{canonical form}) if and only if they are identical.

\end{thm}
\begin{proof}
Applying the Jordan normal form theorem, we can assume that $A=J_{m}\oplus J_n$. 
Utilizing Theorem \ref{T-A Thm} and Lemma \ref{B indecom form}, we can deduce that $B$ must have the following form:
\begin{equation}
\arraycolsep=2.5pt\def\arraystretch{1.2}
\newcommand*{\temp}{\multicolumn{1}{|}{}}
\left[
\begin{array}{cccccccccccccccc}
0 & a_1 & \cdots & \cdot  & a_{m-r} &  \cdots & a_{m-1} & \temp & 0 & 0  & \cdots & 0 & b_{n-r}  & b_{n-r+1} & \cdots & b_{n-1} \\ 
\cline{1-16}
0 & 0 & \cdots & 0 & c_{n-r} &  \cdots & c_{n-1} & \temp & 0 & a_1 & \cdots & \cdot & a_{n-r} & d_{n-r+1} & \cdots & d_{n-1} \\ 
\end{array}
\right].
\end{equation}
where $(b_{n-r}, c_{n-r}) \ne (0, 0)$ and $r$ is the canonical rank of $B$.

To maintain the matrix $A$ unaltered, we can only apply a conjugate transformation to $B$ through a non-singular matrix $X$ 
that commutes with $A$. As per Theorem \ref{T-A Thm}, $X$ must be in the following form:
\begin{equation}\label{X short form}
\arraycolsep=2.5pt\def\arraystretch{1.2}
\newcommand*{\temp}{\multicolumn{1}{|}{}}
\left[
\begin{array}{ccccccccccccccc}
x_0 & x_1 & \cdots & x_{s-1} & x_s & \cdots  & x_{m-1} & \temp & y_0 & y_1 & \cdots & y_{n-1} \\ 
\cline{1-12}
0 & 0 & \cdots & 0 & z_0 & \cdots & z_{n-1}  & \temp & w_0 & w_1 & \cdots & w_{n-1} \\ 
\end{array}
\right].
\end{equation}
with $x_0\ne 0$ and $w_0\ne 0$ and $s=m-n$. And so, $X^{-1}BX$ has the following form:
\[
\arraycolsep=2.5pt\def\arraystretch{1.2}
\newcommand*{\temp}{\multicolumn{1}{|}{}}
\left[
\begin{array}{cccccccccccccccc}
0 & a_1 & \cdots & a_{m-r-1} & a'_{m-r} &  \cdots & a'_{m-1} & \temp & 0 & 0  & \cdots & 0 & b'_{n-r}  & b'_{n-r+1} & \cdots & b'_{n-1} \\ 
\cline{1-16}
0 & 0 & \cdots & 0 & c'_{n-r} &  \cdots & c'_{n-1} & \temp & 0 & a_1 & \cdots & \cdot & a_{n-r} & d'_{n-r+1} & \cdots & d'_{n-1} \\ 
\end{array}
\right],
\]
where $(b'_{n-r},c'_{n-r}) = (b_{n-r}w_0, c_{n-r} w_0^{-1})$. 

If $c_{n-r}\ne 0$, then we can reduce $c_{n-r}$ to 1 by setting $w_0=c_{n-r}$ and then progressively reduce the following elements to $0$:
\[
a_{m-r}, c_{n-r+1}, a_{m-r+1}, c_{n-r+2},\cdots, a_{m-2}, c_{n-1}, a_{m-1}.
\]
When the reduction process stops, $B$ is in the form of $B_{m,n,r}$.

If $c_{n-r} = 0$,  then $b_{n-r}\ne 0$. We can reduce $b_{n-r}$ to 1 by setting $w_0=b_{n-r}^{-1}$ and then progressively reduce the following elements to $0$:
\[
a_{m-r}, b_{n-r+1}, a_{m-r+1}, b_{n-r+2},\cdots, a_{m-2}, b_{n-1}, a_{m-1}.
\]
When the reduction process stops, $B$ is in the form of $B'_{m,n,r}$.

The second part of the statement is a result of the reduction process described above.
\end{proof}

\begin{thm}\label{thm inv}
Consider positive integers $m$ and $n$ with $m>n$, and let
$(A,B)$ be an indecomposable pair of commuting nilpotent matrices of order $(m+n)\times (m+n)$ over $\mathbb{F}$, where $A=J_{m}\oplus J_n$.
Assume that $B\in\mathrm{NilC(A)}$ has canonical rank $r$ in the following form:
\[
\arraycolsep=2.5pt\def\arraystretch{1.2}
\newcommand*{\temp}{\multicolumn{1}{|}{}}
\left[
\begin{array}{cccccccccccccccc}
0 & a_1 & \cdots & \cdot  & a_{m-r} &  \cdots & a_{m-1} & \temp & 0 & 0  & \cdots & 0 & b_{n-r}  & b_{n-r+1} & \cdots & b_{n-1} \\ 
\cline{1-16}
0 & 0 & \cdots & 0 & c_{n-r} &  \cdots & c_{n-1} & \temp & 0 & a_1 & \cdots & \cdot & a_{n-r} & d_{n-r+1} & \cdots & d_{n-1} \\ 
\end{array}
\right],
\]
with $(b_{n-r}, c_{n-r}) \ne (0,0)$.

If $c_{n-r} \ne 0$, then $(J_{m}\oplus J_n, X^{-1}BX)$ is the canonical form for $(A, B)$, where
\[
\arraycolsep=3.8pt\def\arraystretch{1.2}
\newcommand*{\temp}{\multicolumn{1}{|}{}}
X = \left[
\begin{array}{cccccccccccccc}
1 & 0 & \cdots & 0 & 0 & \cdots  & 0 & \temp & a_{m-r} & \cdots & a_{m-1} & 0 & \cdots & 0  \\ 
\cline{1-14}
0 & 0 & \cdots & 0 & 0 & \cdots & 0 & \temp & c_{n-r} & \cdots & c_{n-1} & 0 & \cdots & 0 \\ 
\end{array}
\right].
\]
If $c_{n-r}=0$ and $b_{n-r} \ne 0$, then $(J_{m}\oplus J_n, XBX^{-1})$ is the canonical form for  $(A, B)$,  where 
\[
\arraycolsep=3.8pt\def\arraystretch{1.2}
\newcommand*{\temp}{\multicolumn{1}{|}{}}
X = \left[
\begin{array}{cccccccccccccc}
1 & 0 & \cdots & 0 & 0 & \cdots  & 0 & \temp & 0 & \cdots & 0 & 0 & \cdots & 0  \\ 
\cline{1-14}
0 & 0 & \cdots & 0 & a_{m-r} & \cdots & a_{m-1} & \temp & b_{n-r} & \cdots & b_{n-1} & 0 & \cdots & 0 \\ 
\end{array}
\right].
\]
\end{thm}
\begin{proof}
If $c_{n-r} \ne 0$, then we can define $B'=X^{-1}BX$. 
Since $(A,B)$ is an indecomposable pair with $B$ having canonical rank $r$, Lemma \ref{rank inv} implies 
that $(A,B')$ is also indecomposable and $B'$ has canonical rank $r$. Therefore, $B'$ can be expressed in the following form:
\[
\arraycolsep=2.5pt\def\arraystretch{1.2}
\newcommand*{\temp}{\multicolumn{1}{|}{}}
\left[
\begin{array}{cccccccccccccccc}
0 & a'_1 & \cdots & \cdot  & a'_{m-r} &  \cdots & a'_{m-1} & \temp & 0 & 0  & \cdots & 0 & b'_{n-r}  & b'_{n-r+1} & \cdots & b'_{n-1} \\ 
\cline{1-16}
0 & 0 & \cdots & 0 & c'_{n-r} &  \cdots & c'_{n-1} & \temp & 0 & a'_1 & \cdots & \cdot & a'_{n-r} & d'_{n-r+1} & \cdots & d'_{n-1} \\ 
\end{array}
\right],
\]
with $c'_{n-r} \ne 0$. Since $BX=XB'$, by comparing elements pairwise on both sides, we can first obtain:
\[
(c'_{n-r}, c'_{n-r+1}, \cdots, c'_{n-1}) = (1, 0, \cdots, 0),
\]
and then we can obtain
\[
(a'_1, \cdots, a'_{m-r-1}, a'_{m-r}, \cdots, a'_{m-1}) = (a_1, \cdots, a_{m-r-1}, 0, \cdots, 0).
\]
Therefore, by Theorem \ref{thm can}, $(A, B')$ is the canonical form for $(A,B)$.

If $c_{n-r}=0$ and $b_{n-r} \ne 0$, a similar argument shows that $(A, XBX^{-1})$ is the canonical form for $(A,B)$, using the same reasoning as before.

\end{proof}

\vspace{3mm}
\textbf{\large{Appendix: Canonical forms for the case when $m=6$ and $n=4$}}
\vspace{2mm}

Let $(A,B)$ be an indecomposable pair of commuting nilpotent matrices of order $10\times 10$ over $\mathbb{F}$, where $A=J_{6}\oplus J_4$, and let $r$ be the canonical rank of $B$.
The canonical form for $(A,B)$ can be derived using Theorem \ref{thm inv}.

\textbf{Case 1} ($r=1$). 
\begin{itemize}
\item The commuting nilpotent pair
\[
\left(
J_6\oplus J_4 , 
\left[
\newcommand*{\temp}{\multicolumn{1}{|}{}}
\begin{array}{ccccccccccc}
0 & a_1 & a_2 & a_3 & a_4 & a_5 & \temp & 0 & 0 & 0 & b_3 \\
\cline{1-11}
0 & 0 & 0 & 0 & 0 & c_3 & \temp & 0 & a_1 & a_2 & a_3 \\
\end{array}
\right]
\right),
\]
where $c_3\ne 0$, has the following canonical form:
\[
\left(
J_6\oplus J_4 , 
\left[
\newcommand*{\temp}{\multicolumn{1}{|}{}}
\begin{array}{ccccccccccc}
0 & a_1 & a_2 & a_3 & a_4 & 0 & \temp & 0 & 0 & 0 & b_3c_3 \\
\cline{1-11}
0 & 0 & 0 & 0 & 0 & 1 & \temp & 0 & a_1 & a_2 & a_3 \\
\end{array}
\right]
\right).
\]
\item The commuting nilpotent pair
\[
\left(
J_6\oplus J_4 , 
\left[
\newcommand*{\temp}{\multicolumn{1}{|}{}}
\begin{array}{ccccccccccc}
0 & a_1 & a_2 & a_3 & a_4 & a_5 & \temp & 0 & 0 & 0 & b_3 \\
\cline{1-11}
0 & 0 & 0 & 0 & 0 & c_3 & \temp & 0 & a_1 & a_2 & a_3 \\
\end{array}
\right]
\right),
\]
where $c_3=0$ and $b_3\ne 0$, has the following canonical form:
\[
\left(
J_6\oplus J_4 , 
\left[
\newcommand*{\temp}{\multicolumn{1}{|}{}}
\begin{array}{ccccccccccc}
0 & a_1 & a_2 & a_3 & a_4 & 0 & \temp & 0 & 0 & 0 & 1 \\
\cline{1-11}
0 & 0 & 0 & 0 & 0 & 0 & \temp & 0 & a_1 & a_2 & a_3 \\
\end{array}
\right]
\right).
\]
\end{itemize}
\textbf{Case 2} ($r=2$). 
\begin{itemize}
\item The commuting nilpotent pair
\[
\left(
J_6\oplus J_4 , 
\left[
\newcommand*{\temp}{\multicolumn{1}{|}{}}
\begin{array}{ccccccccccc}
0 & a_1 & a_2 & a_3 & a_4 & a_5 & \temp & 0 & 0 & b_2 & b_3 \\
\cline{1-11}
0 & 0 & 0 & 0 & c_2 & c_3 & \temp & 0 & a_1 & a_2 & d_3 \\
\end{array}
\right]
\right),
\]
where $c_2\ne 0$, has the following canonical form:
\[
\left(
J_6\oplus J_4 , 
\left[
\newcommand*{\temp}{\multicolumn{1}{|}{}}
\begin{array}{ccccccccccc}
0 & a_1 & a_2 & a_3 & 0 & 0 & \temp & 0 & 0 & b'_2 & b'_3 \\
\cline{1-11}
0 & 0 & 0 & 0 & 1 & 0 & \temp & 0 & a_1 & a_2 & d_3 \\
\end{array}
\right]
\right),
\]
where 
\begin{align*}
b'_2 &= b_2c_2 \\
b'_3 &= b_2c_3+b_3c_2+(a_3-d_3)a_4. \\
\end{align*}
\item The commuting nilpotent pair
\[
\left(
J_6\oplus J_4 , 
\left[
\newcommand*{\temp}{\multicolumn{1}{|}{}}
\begin{array}{ccccccccccc}
0 & a_1 & a_2 & a_3 & a_4 & a_5 & \temp & 0 & 0 & b_2 & b_3 \\
\cline{1-11}
0 & 0 & 0 & 0 & c_2 & c_3 & \temp & 0 & a_1 & a_2 & d_3 \\
\end{array}
\right]
\right),
\]
where $c_2=0$ and $b_2\ne 0$, has the following canonical form:
\[
\left(
J_6\oplus J_4 , 
\left[
\newcommand*{\temp}{\multicolumn{1}{|}{}}
\begin{array}{ccccccccccc}
0 & a_1 & a_2 & a_3 & 0 & 0 & \temp & 0 & 0 & 1 & 0 \\
\cline{1-11}
0 & 0 & 0 & 0 & 0 & c'_3 & \temp & 0 & a_1 & a_2 & d_3 \\
\end{array}
\right]
\right),
\]
where 
\begin{align*}
c'_3 &= b_2c_3+(a_3-d_3)a_4. \\
\end{align*}
\end{itemize}

\textbf{Case 3} $(r=3)$.
\begin{itemize}
\item The commuting nilpotent pair
\[
\left(
J_6\oplus J_4 , 
\left[
\newcommand*{\temp}{\multicolumn{1}{|}{}}
\begin{array}{ccccccccccc}
0 & a_1 & a_2 & a_3 & a_4 & a_5 & \temp & 0 & b_1 & b_2 & b_3 \\
\cline{1-11}
0 & 0 & 0 & c_1 & c_2 & c_3 & \temp & 0 & a_1 & d_2 & d_3 \\
\end{array}
\right]
\right),
\]
where $c_1\ne 0$, has the following canonical form:
\[
\left(
J_6\oplus J_4 , 
\left[
\newcommand*{\temp}{\multicolumn{1}{|}{}}
\begin{array}{ccccccccccc}
0 & a_1 & a_2 & 0 & 0 & 0 & \temp & 0 & b'_1 & b'_2 & b'_3 \\
\cline{1-11}
0 & 0 & 0 & 1 & 0 & 0 & \temp & 0 & a_1 & d_2 & a_3+d_3 \\
\end{array}
\right]
\right),
\]
where 
\begin{align*}
b'_1 &= b_1c_1 \\
b'_2 &= b_1c_2+b_2c_1+(a_2-d_2)a_3 \\
b'_3 &= b_1c_3+b_2c_2+b_3c_1+(a_2-d_2)a_4-a_3d_3. \\
\end{align*}
\item The commuting nilpotent pair
\[
\left(
J_6\oplus J_4 , 
\left[
\newcommand*{\temp}{\multicolumn{1}{|}{}}
\begin{array}{ccccccccccc}
0 & a_1 & a_2 & a_3 & a_4 & a_5 & \temp & 0 & b_1 & b_2 & b_3 \\
\cline{1-11}
0 & 0 & 0 & c_1 & c_2 & c_3 & \temp & 0 & a_1 & d_2 & d_3 \\
\end{array}
\right]
\right),
\]
where $c_1=0$ and $b_1\ne 0$, has the following canonical form:
\[
\left(
J_6\oplus J_4 , 
\left[
\newcommand*{\temp}{\multicolumn{1}{|}{}}
\begin{array}{ccccccccccc}
0 & a_1 & a_2 & 0 & 0 & 0 & \temp & 0 & 1 & 0 & 0 \\
\cline{1-11}
0 & 0 & 0 & 0 & c'_2 & c'_3 & \temp & 0 & a_1 & d_2 & a_3+d_3 \\
\end{array}
\right]
\right),
\]
where 
\begin{align*}
c'_2 &= b_1c_2+(a_2-d_2)a_3 \\
c'_3 &= b_1c_3+b_2c_2+(a_2-d_2)a_4-a_3d_3. \\
\end{align*}
\end{itemize}
\textbf{Case 4} ($r=4$). 
\begin{itemize}
\item The commuting nilpotent pair
\[
\left(
J_6\oplus J_4 , 
\left[
\newcommand*{\temp}{\multicolumn{1}{|}{}}
\begin{array}{ccccccccccc}
0 & a_1 & a_2 & a_3 & a_4 & a_5 & \temp & b_0 & b_1 & b_2 & b_3 \\
\cline{1-11}
0 & 0 & c_0 & c_1 & c_2 & c_3 & \temp & 0 & d_1 & d_2 & d_3 \\
\end{array}
\right]
\right),
\]
where $c_0\ne 0$, has the following canonical form:
\[
\left(
J_6\oplus J_4 , 
\left[
\newcommand*{\temp}{\multicolumn{1}{|}{}}
\begin{array}{ccccccccccc}
0 & a_1 & 0 & 0 & 0 & 0 & \temp & b'_0 & b'_1 & b'_2 & b'_3 \\
\cline{1-11}
0 & 0 & 1 & 0 & 0 & 0 & \temp & 0 & d_1 & a_2+d_2 & a_3+d_3 \\
\end{array}
\right]
\right),
\]
where 
\begin{align*}
b'_0 &= b_0c_0 \\
b'_1 &= b_0c_1+b_1c_0+(a_1-d_1)a_2 \\
b'_2 &=  b_0c_2+b_1c_1+b_2c_0+(a_1-d_1)a_3-a_2d_2 \\
b'_3 &= b_0c_3+b_1c_2+b_2c_1+b_3c_0+(a_1-d_1)a_4-a_2d_3-a_3d_2. \\
\end{align*}
\item The commuting nilpotent pair
\[
\left(
J_6\oplus J_4 , 
\left[
\newcommand*{\temp}{\multicolumn{1}{|}{}}
\begin{array}{ccccccccccc}
0 & a_1 & a_2 & a_3 & a_4 & a_5 & \temp & b_0 & b_1 & b_2 & b_3 \\
\cline{1-11}
0 & 0 & c_0 & c_1 & c_2 & c_3 & \temp & 0 & d_1 & d_2 & d_3 \\
\end{array}
\right]
\right),
\]
where $c_0=0$ and $b_0\ne 0$, has the following canonical form:
\[
\left(
J_6\oplus J_4 , 
\left[
\newcommand*{\temp}{\multicolumn{1}{|}{}}
\begin{array}{ccccccccccc}
0 & a_1 & 0 & 0 & 0 & 0 & \temp & 1 & 0 & 0 & 0 \\
\cline{1-11}
0 & 0 & 0 & c'_1 & c'_2 & c'_3 & \temp & 0 & d_1 & a_2+d_2 & a_3+d_3 \\
\end{array}
\right]
\right),
\]
where 
\begin{align*}
c'_1 &= b_0c_1+(a_1-d_1)a_2 \\
c'_2 &=  b_0c_2+b_1c_1+(a_1-d_1)a_3-a_2d_2 \\
c'_3 &= b_0c_3+b_1c_2+b_2c_1+(a_1-d_1)a_4-a_2d_3-a_3d_2. \\
\end{align*}
\end{itemize}

\vspace{0.5cm}
%\newpage
\textit{Email address}: \texttt{jiuzhao.hua@gmail.com}

\end{document}